\documentclass[letterpaper, 10 pt, conference]{IEEEconf}  
\IEEEoverridecommandlockouts
\title{\LARGE \bf
Hybrid Systems as Coalgebras: Lyapunov Morphisms for Zeno Stability
}

\author{Joe Moeller and Aaron D. Ames
\thanks{The authors are with the Department of Mechanical and Civil Engineering, California Institute of Technology, Pasadena, CA. 
{\tt\small \{jmoeller, ames\}@caltech.edu}
This research was supported by the Air Force Office of Scientific Research under the Multidisciplinary University Research Initiative grant Hybrid Dynamics ‐ Deconstruction and Aggregation (HyDDRA). We would like to thank Max de Sa, Álvaro Abella, and Paulo Tabuada for their help.
}%
}

\usepackage{amsthm, amssymb, amsmath, mathtools}
\usepackage{setspace}
\usepackage{xcolor}
\definecolor{darkgreen}{rgb}{0,0.45,0}
\usepackage[
    colorlinks, 
    citecolor=blue, 
    urlcolor=darkgreen, 
    final, 
    hyperindex, 
    linkcolor = blue
]{hyperref}
\usepackage{tikz}
\usepackage{tikz-cd}

\newcommand{\subsect}[1]{\vspace{0.5em}\noindent\textbf{#1.}}

\usepackage[capitalize]{cleveref}
\usepackage{enumerate}

\newtheorem{thm}{Theorem}
\newtheorem{prop}[thm]{Proposition}
\newtheorem{lem}[thm]{Lemma}
\newtheorem{cor}[thm]{Corollary}

\newtheorem{construction}[thm]{Construction}

\theoremstyle{definition}
\newtheorem{defn}[thm]{Definition}
\newtheorem{example}[thm]{Example}

\newtheorem{rmk}[thm]{Remark}

\newcommand{\define}[1]{{\bf \boldmath{#1}}}

\newcommand{\ex}[1]{
\hspace*{\fill}
\textcolor{gray}{#1}
}

\newcommand{\exmath}[1]{
\tag*{
\textcolor{gray}{$#1$}
} 
}

\newcommand{\maps}{\colon}

\newcommand{\id}{\mathrm{id}}

\newcommand{\chart}{\rightrightarrows}

\newcommand{\N}{\mathbb N}
\newcommand{\R}{\mathbb R}
\newcommand{\Rplus}{\R_{\geq 0}}

\newcommand{\K}{\mathcal K}

\newcommand{\namedcat}[1]{\mathsf{#1}}
\newcommand{\C}{\namedcat{C}}

\newcommand{\Chart}{\namedcat{Chart}}

\newcommand{\Man}{\namedcat{Man}}
\newcommand{\Set}{\namedcat{Set}}

\newcommand{\functor}[1]{\mathcal{#1}}
\newcommand{\F}{\functor{F}}
\renewcommand{\H}{\functor{H}}
\renewcommand{\P}{\functor{P}}
\newcommand{\T}{\functor{T}}
\newcommand{\U}[1]{\underline{#1}}

\begin{document}

\maketitle
\thispagestyle{empty}
\pagestyle{empty}

\begin{abstract}
    Hybrid dynamical systems exhibit a diverse array of stability phenomena, each currently addressed by separate Lyapunov-like results. We show that these results are all instances of a single theorem: a Lyapunov function is a morphism from a hybrid system into a simple stable target system $\sigma$, and different stability notions such as Lyapunov stability, asymptotic stability, exponential stability, and Zeno stability correspond to different choices of $\sigma$. This unification is achieved by expressing hybrid systems as coalgebras of an endofunctor $\mathcal H$ on a category $\mathsf{Chart}$ that naturally blends continuous and discrete dynamics.  Instantiating a general categorical Lyapunov theorem for coalgebras to this setting results in new Lypaunov-like conditions for the stability of Zeno equilibria and the existence of Zeno behavior in hybrid systems. 
\end{abstract}

\section{Introduction}
\label{sec:intro}

Hybrid dynamical systems couple continuous flow with discrete jumps \cite{alur1991hybrid, LygerosJohanssonSimic2003, SimicGeomHySys}. They arise throughout robotics and control: bipedal locomotion \cite{spong2005controlled, RESCLF}, robotic manipulation \cite{choset2005principles}, and mechanical systems with impacts \cite{NonsmoothMechanics} all exhibit hybrid behavior. A central challenge in hybrid systems theory is stability analysis, and Lyapunov methods have proven to be the right tool \cite{lamperski2012lyapunov, GoebelTeel08, GoebelSanfeliceTeel12}.

Lyapunov stability theorems for continuous-time and discrete-time systems are clearly analogous. In linear systems both reduce to eigenvalue conditions. Yet the analogous theory for hybrid systems remains less developed. Category theory provides a rigorous language for such analogies, identifying the common structure shared by different classes of system and proving theorems at that level of generality \cite{CatsinContext}. It has already found productive applications in systems and control theory, including signal flow graphs \cite{bonchi2014categorical, baez2014categories}, open dynamical systems \cite{tabuada2005quotients, schultz2020dynamical, CategoricalSystemsTheory}, bisimulation \cite{haghverdi2005bisimulation}, and hybrid systems \cite{lerman2020networks}. The categorical Lyapunov framework of \cite{CLT1, CLT2} axiomatizes Lyapunov stability in an arbitrary category and proves a general theorem from which classical results follow as special cases. In this paper, we instantiate that framework to hybrid systems.

We model hybrid systems as coalgebras for an endofunctor $\H$ on a category $\Chart$ of hybrid state spaces, and derive a unified Lyapunov theorem that simultaneously covers Zeno stability, stability of periodic orbits, and robust stability to set-valued resets. Hybrid automata \cite{alur1991hybrid} have been encoded coalgebraically before \cite{HybridAutomataCoalgebras}, however, this encoding is not suitable for our purposes. A key feature of our framework is the notion of a simulation morphism between hybrid systems, which allows stability results to transfer from simple systems to complex ones. We apply this to Lagrangian hybrid systems \cite{NonsmoothMechanics}, deriving Zeno stability by mapping to the bouncing ball and transferring along a simulation map.

\subsect{Contributions} We make the following contributions.
We introduce the category $\Chart$ and the endofunctor $\H$, and show that $\H$-coalgebras recover classical hybrid systems.
We instantiate the categorical Lyapunov framework of \cite{CLT1, CLT2} to $\H$-coalgebras, yielding a general hybrid Lyapunov theorem.
We apply this to Zeno stability, recovering and extending results of \cite{lamperski2012lyapunov, GoebelTeel08, ames2006stability}, including a new summability bound.
We prove stability of hybrid periodic orbits, recovering the RES result of \cite{RESCLF}.
We introduce simulation morphisms and a transference theorem, which we apply to Lagrangian hybrid systems.

\section{Systems as Coalgebras}
\label{sec:coalgebras}

The goal of this section is to develop the categorical language used throughout this paper. The idea is that various notions of system such as vector field, discrete time dynamical system, and labelled transition system are all instances of a notion from category theory called ``coalgebra'' \cite{rutten2000universal}. Any theory developed at the abstract level of coalgebras holds automatically in all the cases listed above and many more. We seek to include hybrid systems on this list. 

\begin{defn}
\label{def:coalgebra}
    Let $\C$ be a category and $\F \maps \C \to \C$ a functor. An \define{$\F$-coalgebra} is a pair $(X, f)$ consisting of an object $X \in \C$ and a morphism $f \maps X \to \F(X)$. A \define{morphism of $\F$-coalgebras} $\phi \maps (X, f) \to (Y, g)$ is a morphism $\phi \maps X \to Y$ in $\C$ such that the following diagram commutes:
    \begin{equation*}
    \begin{tikzcd}
        X \arrow[r, "\phi"] \arrow[d, "f"'] 
        & 
        Y 
        \arrow[d, "g"] 
        \\
        \F(X) 
        \arrow[r, "\F(\phi)"'] 
        & 
        \F(Y)
    \end{tikzcd}
    \ g(\phi(x))=\F(\phi)(f(x)) \ \forall x \in X
    \end{equation*}
    Let $\mathbf{Coalg}_\F$ denote the category of $\F$-coalgebras.
\end{defn}

\begin{example}
\label{ex:Tcoalgebra}
    Let $\Man$ denote the category of smooth manifolds with boundary and smooth maps, and let $\T \maps \Man \to \Man$ denote the tangent bundle functor. A $\T$-coalgebra is a smooth map $f \maps M \to \T M$. A vector field on $M$ is a $\T$-coalgebra which is a section of the canonical tangent bundle projection $\pi \maps \T M \to M$, i.e.\ satisfying $\pi \circ f = \mathrm{id}_M$.
\end{example}

\begin{example}
\label{ex:Pcoalgebra}
    Let $\Set$ denote the category of sets and functions, and let $\P \maps \Set \to \Set$ denote the covariant powerset functor. A $\P$-coalgebra is a map $f \maps X \to \P(X)$, which can be interpreted as a transition system: $f(x)$ is the set of states reachable from $x$ in one step.
\end{example}

\subsect{Hybrid Systems}
We assume the reader is familiar with the basic notions of hybrid systems theory.  We recall the following definitions in order to establish notation. 

\begin{defn}[\cite{lamperski2012lyapunov}, Def.\ 1]
\label{def:classichybrid}
    A \define{hybrid system} is a tuple $H = (\Gamma, D, G, R, F)$, where $\Gamma = (V,E)$ is a directed graph with vertices $V$ (discrete modes), edges $E$ (transitions), and source and target maps $s,t \maps E \to V$; $D = \{D_v\}_{v \in V}$ is a collection of smooth manifolds; $G = \{G_e\}_{e \in E}$ is a collection of guards $G_e \subseteq D_{s(e)}$; $R = \{R_e\}_{e \in E}$ is a collection of smooth reset maps $R_e \maps G_e \to D_{t(e)}$; and $F = \{f_v\}_{v \in V}$ is a collection of vector fields $f_v \maps D_v \to \T D_v$.
\end{defn}

The conceptual complexity of hybrid systems that makes their analysis trickier than other types of system continues into this categorical approach as well. As such, in order to realize hybrid systems as coalgebras, we introduce a functor $\H$ which naturally blends the functors $\T$ and $\P$ of Examples~\ref{ex:Tcoalgebra} and~\ref{ex:Pcoalgebra}. This requires a base category $\Chart$\footnote{The term ``chart'' is borrowed from \cite{CategoricalSystemsTheory} where monoidal categories of charts are constructed as a Grothendieck construction \cite{MonGroth}.} which is similarly a blend of $\Set$ and $\Man$. Note that sometimes a category is named after its morphisms rather than its objects. For a manifold $M$, let $\underline M$ denote the underlying set of $M$. 

\begin{defn}
\label{def:chart}
    Let $\Chart$ denote the following category. An object of $\Chart$ is a pair $(M, S)$, which we will usually denote by $\binom SM$, where $M$ is a manifold, and $S$ is a set. 
    A \define{chart} from $(M,S)$ to $(M',S')$, denoted by \[\binom{f_d}{f_c} \maps \binom SM \chart \binom{S'}{M'},\] is a pair of maps: $f_c \maps M \to M'$ smooth, and $f_d \maps S \times \U M \to S'$ function.
    The composite of two charts $\binom{f_d}{f_c} \maps \binom SM \chart \binom{S'}{M'}$ and $\binom{g_d}{g_c} \maps \binom{S'}{M'} \chart \binom{S''}{M''}$ is given by the composites:
    \begin{align*}
        S \times \underline M \xrightarrow{\id_S \times \Delta_M} S \times \underline M \times &\underline M \xrightarrow{f_d \times \underline{f_c}} S' \times \underline M' \xrightarrow{g_d} S''
        \\
        M \xrightarrow{f_c} &M' \xrightarrow{g_c} M''
    \end{align*}
    or more compactly:
    \[\binom{g_d \circ (f_d \times \underline{f_c}) \circ (\id_S \times \Delta_M)}{g_c \circ f_c}.\]
\end{defn}

\begin{prop}
\label{prop:Chartproducts}
    The object $\binom{1}{1}$ is a terminal object in $\Chart$. The product of $\binom SM$ and $\binom{S'}{M'}$ in $\Chart$ is $\binom{S \times S'}{M \times M'}$.
\end{prop}
\begin{proof}
    A chart into $\binom{1}{1}$ is the pair of unique maps $! \maps M \to 1$ and $! \maps S \times \underline{M} \to 1$ into the terminal objects of $\Man$ and $\Set$ respectively. The projection from $\binom{S \times S'}{M \times M'}$ onto $\binom SM$ is given by $\pi_M \maps M \times M' \to M$ in $\Man$ and $\pi_S \maps S \times S' \times \underline{M} \times \underline{M'} \to S$ in $\Set$. The other projection is given similarly, and it is straightforward to check that these satisfy the universal property of products.
\end{proof}

\begin{defn}
\label{def:H}
    The endofunctor $\H \maps \Chart \to \Chart$ is defined as follows. On objects: $\H \binom SM \coloneqq \binom{\P(S \times \underline{M})}{\T M}$.
    On a morphism $f = \binom{f_d}{f_c} \maps \binom SM \chart \binom{S'}{M'}$, define $\H f \maps \H\binom SM \chart \H\binom{S'}{M'}$ to have component maps
    \begin{align*}
        (\H f)_c &\maps \T M \to \T M', \\
        (\H f)_d &\maps \P(S \times \underline{M}) \times \underline{\T M} \to \P(S' \times \underline{M'}),
    \end{align*}
    given by $(\H f)_c = \T f_c$ and
    \[
        (\H f)_d = \P((f_d \times \underline{f_c}) \circ (\id \times \Delta_{\underline{M}})) \circ \pi_{\P(S \times \underline{M})}.
    \]
\end{defn}

\begin{prop}
\label{prop:Hfunctor}
    $\H$ is an endofunctor on $\Chart$.
\end{prop}
\begin{proof}
    Associativity follows from the coassociativity of diagonal. Unitality follows from the unique map into the terminal object being counital to the diagonal map.
\end{proof}

An $\H$-coalgebra is a chart of the form
\[
    \binom{f_d}{f_c} \maps \binom SM \chart \H\binom SM = \binom{\P(S \times \underline{M})}{\T M},
\]
which consists of the following data: a smooth manifold $M$, a set $S$, a smooth map $f_c \maps M \to \T M$, and a function $f_d \maps S \times \underline{M} \to \P(S \times \underline{M})$. The continuous component $f_c$ specifies the continuous dynamics at each point of $M$. The discrete component $f_d$ assigns to each pair $(s, x) \in S \times \underline{M}$ a set $f_d(s, x) \subseteq S \times \underline{M}$ of possible next states after a jump. When $f_d(s, x) = \emptyset$, no jump is available and the system flows according to $f_c$. When $f_d(s, x)$ is nonempty, the system may jump to any $(s', x') \in f_d(s, x)$, instantaneously updating both the discrete mode and the continuous state.

This structure simultaneously encodes the continuous dynamics, the guard set (where $f_d$ is nonempty), and the reset map (the values in $f_d(s,x)$). The class of $\H$-coalgebras is broader than classical hybrid systems, as $S$ need not carry a graph structure, jumps need not be deterministic, and the guard need not be a submanifold.

\begin{construction}
\label{con:hybridcoalgebra}
    Given a hybrid system $H = (\Gamma, D, G, R, F)$ as in Def.~\ref{def:classichybrid}, we define an $\H$-coalgebra $\binom{f_d}{f_c} \maps \binom SM \chart \H\binom SM$ as follows:
    \begin{itemize}
        \item $S = V$, the vertex set of $\Gamma$,
        \item $M = \bigsqcup_{v \in V} D_v$, the disjoint union of all domains,
        \item $f_c = \bigsqcup_{v \in V} f_v \maps \bigsqcup_{v \in V} D_v \to \T M \cong \bigsqcup_{v \in V} \T D_v$, the vector fields assembled over all modes,
        \item Let $E_v = \{e \in E \mid s(e) = v\}$ be the set of outgoing edges from $v$. Then define $f_d \maps V \times \underline{M} \to \P(V \times \underline{M})$ :
        \[
            f_d(v, x) = \begin{cases} \bigcup_{e \in E_v} \{(t(e), R_e(x))\} & x \in \bigcup_{e \in E_v} G_e, \\ \emptyset & \text{otherwise.} \end{cases}
        \]
    \end{itemize}
    We assume that the guards $\{G_e\}_{e \in E_v}$ are disjoint for each  $v \in V$, so that $f_d(v,x)$ is either empty or a singleton.
\end{construction}

\begin{rmk}
\label{rmk:switchingsystems}
    A switching system, consisting of a family of vector fields $\{f_p\}_{p \in P}$ on a manifold $M$ with a switching signal selecting which vector field is active \cite{liberzon2003switching}, is also naturally encoded as an $\H$-coalgebra. It corresponds to the special case where $S = P$, $M$ is unchanged, $f_c = f_p$ for the currently active mode, and
    \[
        f_d(p, x) = \{(p', x) : p' \in P\}
    \]
    for all $p \in P$ and $x \in \underline{M}$. Two features distinguish this from Construction~\ref{con:hybridcoalgebra}: jumps are available at every point in the state space rather than only on a guard, and the continuous state is unchanged at each jump since the reset is the identity. The Lyapunov conditions of Section~\ref{sec:lyapunov} applied to this encoding recover the common Lyapunov function conditions for switching systems, as discussed in Section~\ref{sec:lyapunov}.
\end{rmk}

\subsect{Solutions}
A solution curve of an ODE $\dot x=f(x)$ on domain $X$ is a continuously differentiable map $\gamma \maps I \to X$ for some interval $I$ such that $\frac{d}{dt}\gamma(t)=f(\gamma(t))$. This equation can be expressed as a commutative diagram.
\[
\begin{tikzcd}
    I
    \arrow[r, "\gamma"]
    \arrow[d, "1_I"']
    &
    X
    \arrow[d, "f"]
    \\
    \T(I)
    \arrow[r, "\T(\gamma)"']
    &
    \T(X)
\end{tikzcd}
\]
Here $1_I$ denotes the constant $1$ vector field on $I$. Note that this diagram says that $\gamma$ is a morphism of $\T$-coalgebras $1_I \to f$, per Def.~\ref{def:coalgebra}. In the coalgebraic framework, a solution to an $\F$-coalgebra should be a morphism of coalgebras from a ``time domain object'' equipped with a ``unit clock coalgebra'' into the system of interest. 
We recall now the definition of an ``execution'' of a hybrid system, the analog of a solution curve.

\begin{defn}[\cite{lamperski2012lyapunov}, Def.\ 2]
\label{def:execution}
    A \define{hybrid system execution} of $H$ is a tuple $\chi = (\Lambda, I, \rho, C)$ where $\Lambda \subseteq \N$ is an index set, $I = \{I_j = [\tau_j, \tau_{j+1}]\}_{j \in \Lambda}$ is a collection of intervals, $\rho \maps \Lambda \to V$ assigns a discrete mode to each interval, and $C = \{c_j \maps I_j \to D_{\rho(j)}\}_{j \in  \Lambda}$ is a collection of continuous trajectories satisfying, for all $j, j+1 \in \Lambda$:
    \begin{equation}
    \label{classical_condition}
    \begin{split}
    \frac{d}{dt} c_j(t) = & f_{\rho(j)}(c_j(t)), \\
     c_j(\tau_{j+1}) \in & G_{(\rho(j), \rho(j+1))}, \\
     c_{j+1}(\tau_{j+1}) = & R_{(\rho(j), \rho(j+1))}(c_j(\tau_{j+1})).
    \end{split}
    \end{equation}
\end{defn}

\begin{defn}
\label{def:clockcoalgebra}
    Let $\Lambda$ be either $\N$ or a finite set $\{0, \ldots, n-1\}$ with its usual order, and let $\{I_j = [\tau_j, \tau_{j+1}]\}_{j \in \Lambda}$ be a collection of intervals. The \define{unit clock coalgebra} on $\binom{\Lambda}{\bigsqcup_\Lambda I_j}$ is the $\H$-coalgebra $\binom{1_d}{1_c} \maps \binom{\Lambda}{\bigsqcup_\Lambda I_j} \rightrightarrows \H\binom{\Lambda}{\bigsqcup_\Lambda I_j}$ with components:
    \begin{itemize}
        \item $1_c \maps \bigsqcup_\Lambda I_j \to \T\bigsqcup_\Lambda I_j$ the constant unit vector field on each component interval,
        \item $1_d \maps \Lambda \times \underline{\bigsqcup_\Lambda I_j} \to \P\left(\Lambda \times \underline{\bigsqcup_\Lambda I_j}\right)$ defined by
        \[
            1_d(j, k, t) = \begin{cases} \{(j+1, j+1, \tau_{j+1})\} & j = k,\ t = \tau_{j+1}, \\ \emptyset & \text{otherwise,} \end{cases}
        \]
        where we denote elements of $\Lambda \times \underline{\bigsqcup_\Lambda I_j}$ by triples $(j, k, t)$ with $j, k \in \Lambda$ and $t \in I_k$.
    \end{itemize}
\end{defn}

A \define{solution} to an $\H$-coalgebra $\binom{f_d}{f_c} \maps \binom SM \rightrightarrows \H\binom SM$ is a morphism of $\H$-coalgebras
\[
    \binom{\psi_d}{\psi_c} \maps \binom{\Lambda}{\bigsqcup_\Lambda I_j} \rightrightarrows \binom SM
\]
from a unit clock coalgebra into the system. Unpacking the coalgebra morphism condition, this requires:
\begin{equation}
    \frac{\partial \psi_c}{\partial t}(j, t) = f_c(\psi_c(j, t)), \label{eq:solutionflow} 
\end{equation}
\begin{align}
    &f_d(\psi_d(j, k, t), \psi_c(k, t)) = \label{eq:solutiondiscrete}
    \\&  \qquad  \begin{cases} \{(\psi_d(j+1, j+1, \tau_{j+1}), & j = k,\ t = \tau_{j+1}, 
    \\ \qquad\psi_c(j+1, \tau_{j+1}))\}
    \\ 
    \emptyset & \text{otherwise.} \end{cases} \nonumber
\end{align}

\begin{lem}
\label{lem:solutionsequiv}
    Let $\binom{f_d}{f_c} \maps \binom SM \rightrightarrows \H\binom SM$ be an $\H$-coalgebra obtained via Construction~\ref{con:hybridcoalgebra}. Then a chart $\binom{\psi_d}{\psi_c} \maps \binom{\Lambda}{\bigsqcup_\Lambda I_j} \rightrightarrows \binom SM$ is an $\H$-coalgebra morphism if and only if the conditions~\eqref{classical_condition} are satisfied.
\end{lem}
\begin{proof}
    Define $\rho(j) := \psi_d(j, j, \tau_j)$, so $\rho \colon \Lambda \to V$ assigns a discrete mode to each interval.
    Condition \eqref{eq:solutionflow} is exactly \eqref{classical_condition} for $\psi_c$. 
    At $j = k$ and $t = \tau_{j+1}$, \eqref{eq:solutiondiscrete} and Construction~\ref{con:hybridcoalgebra} gives:
    \begin{align*}
        \{(\psi_d(j+1, j+1, \tau_{j+1}), &\psi_c(j+1, \tau_{j+1}))\} 
        \\&= \{(t(e), R_e(\psi_c(j, \tau_{j+1})))\}
    \end{align*}
    for the edge $e = (\rho(j), \rho(j+1)) \in E$, which gives $\psi_c(j+1, \tau_{j+1}) = R_e(\psi_c(j, \tau_{j+1}))$, the third equation of~\eqref{classical_condition}. The condition $\psi_c(j, \tau_{j+1}) \in G_e$ follows from the fact that $f_d$ is nonempty only on the guard, which is the second equation of~\eqref{classical_condition}. Otherwise $f_d = \emptyset$ and no jump occurs. Since the guards are pairwise disjoint for edges sharing a source, $f_d(v,x)$ is either empty or a singleton, so the set equality in \eqref{eq:solutiondiscrete} reduces to the pointwise conditions of \eqref{classical_condition}. The converse is immediate: given an execution $(\Lambda, I, \rho, C)$ satisfying~\eqref{classical_condition}, setting $\psi_c(j,t) = c_j(t)$ and $\psi_d(j,k,t) = \rho(j)$ yields an $\mathcal{H}$-coalgebra morphism.
\end{proof}


\subsect{Generalized Elements and Forward Invariance}
In category theory, points $x \in E$ are conflated with maps of the form $x \maps 1 \to E$ whose image is $\{x\}$. A point is the simplest thing that can carry the property of being an equilibrium. We can extend this to a more general notion by replacing the terminal object $1$ with an arbitrary object $Z$. A \define{generalized element} of an object $E$ is any morphism $z^* \maps Z \to E$. All points are generalized elements of the specific form $Z=1$. A subspace $U \subseteq E$ is also a generalized element $U \hookrightarrow E$. 

\begin{defn}
\label{def:forwardinvariant}
    A generalized element $z^* \maps Z \to E$ is \define{forward invariant} if $\forall$ solutions $c \maps I \to E$ such that $\exists$ a point $y \maps 1 \to Z$ with $c \circ 0_I = z^* \circ y$, $\exists$ a morphism $c_Z \maps I \to Z$ such that the following diagram commutes:
    \[
    \begin{tikzcd}
        1 \arrow[r, "y"] \arrow[d, "0_I"'] & Z \arrow[d, "z^*"] \\
        I \arrow[ur, dashed, "\exists c_Z" description] \arrow[r, "c"'] & E
    \end{tikzcd}
    \]
    That is, if a solution starts in the image of $z^*$, it remains in the image of $z^*$ for the rest of the time.
\end{defn}

When $Z = 1$ this reduces to the definition of an equilibrium point, with $z^* = x^*$, $y = \mathrm{id}$, and $c_Z = !$ the unique morphism to $1$. In $\Chart$, a generalized element of $\binom{S}{M}$ is a chart $\binom{z_d}{z_c} \maps \binom{Z_d}{Z_c} \rightrightarrows \binom{S}{M}$, consisting of a smooth map $z_c \maps Z_c \to M$ and a function $z_d \maps Z_d \times \underline{Z_c} \to S$.

\subsect{Zeno Stability}
An execution $\chi = (\Lambda, I, \rho, C)$ is a \define{Zeno execution} if
$\Lambda = \N$ and the \define{Zeno time} is finite:
\[
    \tau_\infty := \lim_{j \to \infty} \tau_j = \tau_0 + \sum_{j=1}^{\infty} (\tau_j - \tau_{j-1}) < \infty.
\]

\begin{defn}[\cite{lamperski2012lyapunov}, Def.\ 4]
\label{def:ZenoEquilibrium}
    A \define{Zeno equilibrium} of a hybrid system $H = (\Gamma, D, G, R, F)$ is a collection of points $z = \{z_v\}_{v \in V}$ with $z_v \in D_v$, satisfying:
    \begin{enumerate}
        \item \textbf{Non-Equilibrium:} $f_v(z_v) \neq 0$ for all $v \in V$.
        \item \textbf{Jump Invariance:} $z_v \in G_e$ and $R_e(z_v) = z_{v'}$ for all $e = (v, v') \in E$.
    \end{enumerate}
\end{defn}

\begin{example}
\label{ex:zenogeneralized}
    A Zeno equilibrium $z = \{z_v\}_{v \in V}$ of a hybrid system is a forward-invariant generalized element of the $\H$-coalgebra obtained via Construction~\ref{con:hybridcoalgebra}. The corresponding chart is $\binom{z_d}{z_c} \maps \binom{1}{V} \rightrightarrows \binom{V}{M}$ where $z_c \maps V \to M$ (with $V$ viewed as a discrete manifold) is given by $z_c(v) = z_v$, and $z_d \maps 1 \times V \cong V \to V$ is the canonical isomorphism. Forward invariance holds since any execution starting at some $z_v$ remains confined to the set $\{z_v\}_{v \in V}$ under repeated jumps.
\end{example}

\subsect{Hybrid Periodic Orbits}
Periodic orbits are central to the analysis of hybrid systems such as bipedal walking robots \cite{spong2005controlled, RESCLF}, where a stable gait corresponds to a stable periodic orbit of the underlying hybrid system.

\begin{defn}[Hybrid Periodic Orbit]
\label{def:hybridperiodicorbit}
    Let $\psi \maps \binom{\N}{\bigsqcup_\N I_j} \rightrightarrows \binom{S}{M}$ be a solution to an $\H$-coalgebra. For $K \in \N$ and $T \in \R_{>0}$, we say $\psi$ is \define{$(K,T)$-periodic} if for all $(k,t)$:
    \begin{align*}
        \psi_c(k+K, t+T) &= \psi_c(k,t), \\
        \psi_d(j+K, k+K, t+T) &= \psi_d(j,k,t).
    \end{align*}
    This implies $\tau_{j+K} = \tau_j + T$. The \define{hybrid periodic orbit} $\bar\psi$ is the forward-invariant generalized element $\bar{\psi} \maps \binom{[K]}{\bigsqcup_{j \in [K]} I_j} \rightrightarrows \binom{S}{M}$ given by restricting $\psi$ to its initial period, and the distance to the orbit is:
    \begin{equation}
        \|x\|_{\bar\psi} := \inf_{(k,t) \in [K] \times \bigsqcup_{[K]} I_j} d(x, \psi_c(k,t)).
    \end{equation}
\end{defn}

\section{Categorical Lyapunov Theory}
\label{sec:lyapunov}

Everything thus far has been expressible purely with equations, and thus (strictly) commutative diagrams. Lyapunov theory demands inequalities. In this section we give the categorical approach to Lyapunov theory, and so we begin with our incorporation of partial orders via ``posetal objects''.

\begin{defn}
\label{def:posetal}
    An object $R \in \C$ is \define{posetal} if each hom-set $\C(X, R)$ carries a partial order, denoted $f \Rightarrow g$ for $f \geq g$, such that for any $h \maps X \to Y$, if $g_1 \Rightarrow g_2 \maps Y \to R$ then $g_1 \circ h \Rightarrow g_2 \circ h$.
\end{defn}

Just as a diagram commutes when two paths are equal, we say a diagram \define{lax commutes} when the two paths satisfy an inequality as indicated by a $\Rightarrow$ symbol, e.g.\ \eqref{eqn:lyapunov}.

\begin{example}
\label{ex:posetal2}
    The real line $\R$ with its usual ordering is a posetal object in $\Set$, with $f \leq g \maps X \to \R$ defined pointwise: $f(x) \leq g(x)$ for all $x \in X$. More generally, any partially ordered set carries a pointwise posetal structure.
    
    Two natural choices of posetal object in $\Chart$ are relevant to this paper. The first is $\binom{1}{\R_{\geq 0}}$, where the discrete component is trivial and the order is determined entirely by the continuous component $f_c \leq g_c \maps M \to \R_{\geq 0}$ pointwise. All charts into $\binom{1}{\R_{\geq 0}}$ necessarily have the unique function into the terminal set as the discrete component. The second is $\binom{\R_{\geq 0}}{\R_{\geq 0}}$, where both components are ordered pointwise: $f \leq g \maps \binom{S}{M} \rightrightarrows \binom{\R_{\geq 0}}{\R_{\geq 0}}$ when $f_c \leq g_c \maps M \to \R_{\geq 0}$ and $f_d \leq g_d \maps S \times \underline{M} \to \R_{\geq 0}$ pointwise. 
    
    The functor $\H$ produces objects of the form $\binom{\P(S \times \underline{M})}{\T M}$, and we need posetal structures on these as well. We order $\T R_c$ fiberwise: two tangent vectors over the same base point $a \in R_c$ are compared by their components in $\T_a R_c \cong \R$. We order $\P(R_d \times \underline{R_c})$ by the \define{Hoare order}: $A \leq B$ iff for every $a \in A$ there exists $b \in B$ with $a \leq b$. When $B = \emptyset$, the condition holds vacuously only if $A = \emptyset$; when $A = \emptyset$, the condition holds for any $B$. 
\end{example}

\begin{defn}
\label{def:classK}
    Let $R \in \C$ be a posetal object with base point $0_R \maps 1 \to R$. A morphism $\alpha \maps R \to R$ is \define{class $\mathcal{K}$} if:
    \begin{enumerate}
        \item $\alpha$ is order-preserving: $f \Rightarrow g$ implies $\alpha \circ f \Rightarrow \alpha \circ g$,
        \item $\alpha$ has an order-preserving inverse $\alpha^{-1}$,
        \item $\alpha \circ 0_R = 0_R$.
    \end{enumerate}
\end{defn}

In $\Chart$, a class $\mathcal{K}$ morphism on a posetal object $\binom{R_d}{R_c}$ is a chart $\binom{\alpha_d}{\alpha_c} \maps \binom{R_d}{R_c} \rightrightarrows \binom{R_d}{R_c}$ where $\alpha_c \maps R_c \to R_c$ is a smooth classical class $\mathcal{K}$ function \cite{Khalil} (smooth, strictly increasing, and zero at zero) and $\alpha_d \maps R_d \times \underline{R_c} \to R_d$ is order-preserving in both arguments with an order-preserving inverse and $\alpha_d(0, 0) = 0$. If $\alpha_d$ depends only on its first argument, this reduces to a class $\mathcal{K}$ function on $R_d$.

\begin{defn}
\label{def:semimetric}
    Let $\C$ be a category with finite products and a posetal object $R$ with a base point $0_R$. A \define{semi-metric} on an object $E \in \C$ is a morphism $d \maps E \times E \to R$ satisfying:
    \begin{enumerate}
        \item $d \geq 0_R$, i.e.\ $d(x, y) \geq 0_R$ for all $x, y$,
        \item $\ker(d) \cong \Delta \maps E \to E \times E$, i.e.\ $d(x, y) = 0_R$ if and only if $x = y$.
    \end{enumerate}
    Given a semi-metric $d$ and a generalized element $z^* \maps Z \to E$, the \define{semi-norm} to $z^*$ is the morphism $\|\cdot\|_{z^*} \maps E \to R$ defined by $\|x\|_{z^*} = \inf_{z \in Z} d(x, z^*(z))$, where the infimum is taken in the posetal object $R$. See \cite{CLT2} for more on suprema/infima in posetal objects.
\end{defn}

\begin{defn}
\label{def:positivedefinite}
    A morphism $V \maps E \to R$ is \define{positive definite} with respect to a generalized element $z^* \maps Z \to E$ if there exist class $\mathcal{K}$ morphisms $\underline{\alpha}, \overline{\alpha} \maps R \to R$ such that the following diagram lax commutes:
    \[\begin{tikzcd}
        &
        R
        \arrow[dr, "\overline \alpha"]
        \\
        E
        \arrow[ur, "\|\cdot\|_{z^*}"]
        \arrow[rr, ""{name = V}, ""'{name = B}, "V"description]
        \arrow[dr, "\|\cdot\|_{z^*}"']
        &&
        R
        \\&
        R
        \arrow[ur, "\underline \alpha"']
        \arrow[from = 1-2, to = V, Rightarrow, ""]
        \arrow[from = B, to = 3-2, Rightarrow, ""]
    \end{tikzcd}
    \exmath{\underline{\alpha}(\| x \|_{z^*}) \leq V(x) \leq \overline{\alpha}(\| x\|_{z^*})}
    \]
\end{defn}

\begin{defn}
\label{def:stable}
    A forward-invariant generalized element $z^* \maps Z \to E$ is \define{stable} if there exists a class $\mathcal{K}$ morphism $\alpha \maps R \to R$ such that for all solutions $c \maps I \to E$, the following diagram lax commutes:
    \[
    \begin{tikzcd}[column sep = small]
        1 \arrow[r, "c_0"] \arrow[d, "!"'] & E \arrow[r, "\|\cdot\|_{z^*}"] & R
        \arrow[dll, Rightarrow]\arrow[d, "\alpha"] \\
        I
        \arrow[r, "c"']
        & E
        \arrow[r, "\|\cdot\|_{z^*}"']
        & R
    \end{tikzcd}
    \exmath{\|c(t)\|_{z^*} \leq \alpha(\|c(0)\|_{z^*}) \ \forall t \in I}
    \]
\end{defn}

\subsect{The Categorical Lyapunov Theorem}
We now have the structure necessary to state the general categorical Lyapunov theorem. We first need the notion of a ``measurement object'', which carries the ingredients needed to formulate stability.

\begin{defn}
\label{def:measurementobject}
    Let $\C$ be a category with finite products, $\F \maps \C \to \C$ an endofunctor, and $\mathcal{I}$ a collection of time domain objects each equipped with an initial point $0_I \maps 1 \to I$. A \define{measurement object} is an object $R \in \C$ equipped with a posetal structure, a base point $0_R \maps 1 \to R$, and an $\F$-coalgebra $\sigma \maps R \to \F(R)$ with a posetal structure on $\F(R)$, such that the \define{comparison property} holds: if the diagrams
    \begin{equation}
    \label{eq:comparisonproperty}
    \begin{tikzcd}[column sep = small]
        I
        \arrow[r, "\psi"]
        \arrow[d, "1_I"']
        &
        R
        \arrow[d, "\sigma"]
        \\
        \F(I)
        \arrow[r, "\F\psi"']
        &\F(R)
    \end{tikzcd}
    \
    \begin{tikzcd}[column sep = small]
        I
        \arrow[r, "\phi"]
        \arrow[d, "1_I"']
        &
        R
        \arrow[d, "\sigma"]
        \arrow[dl, Rightarrow]
        \\
        \F(I)
        \arrow[r, "\F\phi"']
        &\F(R)
    \end{tikzcd}
    \
    \begin{tikzcd}
        1
        \arrow[r, "0_I"]
        \arrow[d, "0_I"']
        &
        I
        \arrow[d, "\psi"]
        \arrow[dl, Rightarrow]
        \\
        I
        \arrow[r, "\phi"']
        &
        R
    \end{tikzcd}
    \end{equation}
    (lax) commute, meaning $\psi$ is a true solution, $\phi$ is a sub-solution, and $\phi(0_I) \leq \psi(0_I)$, then $\phi \leq \psi$.
\end{defn}

\begin{defn}[\cite{CLT2}]
\label{def:settingforstability}
    A \define{setting for dynamic stability} is a category $\C$ with finite products equipped with an endofunctor $\F \maps \C \to \C$, a collection $\mathcal{I}$ of time domain objects, a measurement object $R$, and a plant object $E$ with a semi-metric $d \maps E \times E \to R$.
\end{defn}

\begin{thm}[Categorical Lyapunov Theorem \cite{CLT2}]
\label{thm:CLT}
    Consider a setting for dynamic stability. Let $z^* \maps Z \to E$ be a forward-invariant generalized element in a system $f \maps E \to \F(E)$. If $V \maps E \to R$ is positive definite with respect to $z^*$ and the following diagram lax commutes, then $z^*$ is stable. 
    \begin{eqnarray}
    \label{eqn:lyapunov}
    \begin{tikzcd}
        E
        \arrow[r, "V"]
        \arrow[d, "f"']
        &
        R
        \arrow[d, "\sigma"]
        \arrow[dl, Rightarrow]
        \\
        \F(E)
        \arrow[r, "\F(V)"']
        &
        \F(R)
    \end{tikzcd}
    \hspace{1cm} \ex{\frac{\partial V}{\partial x} f(x) \leq \sigma(V(x))}
    \end{eqnarray}
\end{thm}

\subsect{CLT for $\H$-coalgebras}
We now instantiate Thm.~\ref{thm:CLT} to the setting of $\H$-coalgebras. The time domains are the hybrid time domain objects $\binom{\Lambda}{\bigsqcup_\Lambda I_j}$ with unit clock coalgebras from Def.~\ref{def:clockcoalgebra}, and the plant is an $\H$-coalgebra $\binom{f_d}{f_c} \maps \binom{S}{M} \chart \H\binom{S}{M}$. Unpacking the lax commutativity condition of Thm.~\ref{thm:CLT} in $\Chart$ for a Lyapunov morphism $\binom{V_d}{V_c} \maps \binom{S}{M} \chart \binom{R_d}{R_c}$ yields two conditions: a flow condition on $V_c$ and a jump condition coupling both components.

\begin{thm}[CLT for $\H$-coalgebras]
\label{thm:CLT_hybrid}
    Consider a setting for dynamic stability with $\H$-coalgebra $\binom{f_d}{f_c} \maps \binom{S}{M} \chart \H\binom{S}{M}$, measurement object $\binom{\sigma_d}{\sigma_c} \maps \binom{R_d}{R_c} \chart \H\binom{R_d}{R_c}$, and forward-invariant generalized element $z^* \maps \binom{Z_d}{Z_c} \chart \binom{S}{M}$. Suppose there exist class $\K$ morphisms $\underline{\alpha}, \overline{\alpha} \maps \binom{R_d}{R_c} \chart \binom{R_d}{R_c}$ and a chart $\binom{V_d}{V_c} \maps \binom{S}{M} \chart \binom{R_d}{R_c}$ satisfying:

    \textbf{Positive definiteness.}
    \begin{align}
        \underline{\alpha}_c(\|x\|^c_{z^*}) \leq V_c(x) &\leq \overline{\alpha}_c(\|x\|^c_{z^*}), \label{eq:pd_c} \\
        \underline{\alpha}_d(\|(s,x)\|^d_{z^*}, \|x\|^c_{z^*}) \leq V_d(s,x) &\leq \overline{\alpha}_d(\|(s,x)\|^d_{z^*}, \|x\|^c_{z^*}). \nonumber
    \end{align}
    \vspace{-6mm}

    \textbf{Flow condition.}
    \begin{equation}
        \frac{dV_c}{dt} \cdot f_c(x) \leq \sigma_c(V_c(x)). \label{eq:flow}
    \end{equation}

    \textbf{Jump condition.} 
    For all $(s,x) \in S \times \underline{M}$:
    \begin{align}
        \{(V_d(s',x'), V_c(x')) &\mid (s',x') \in f_d(s,x)\} \nonumber
        \\&\leq \sigma_d(V_d(s,x), V_c(x)). 
        \label{eq:jump}
    \end{align}

    Then $z^*$ is stable: there exists a class $\K$ morphism $\alpha \maps \binom{R_d}{R_c} \chart \binom{R_d}{R_c}$ s.t. for any solution $\binom{\gamma_d}{\gamma_c} \maps \binom{\Lambda}{\coprod_\Lambda I_j} \chart \binom{S}{M}$:
    \begin{align}
        \|\gamma_c(j,t)\|^c_{z^*} &\leq \alpha_c(\|\gamma_c(0,0)\|^c_{z^*}), \label{eq:stab_c} \\
        \|(\gamma_d(j,j,t), \gamma_c(j,t))\|^d_{z^*} &\leq \alpha_d(\|\gamma_d(0,0,0)\|^d_{z^*}, \|\gamma_c(0,0)\|^c_{z^*}). \nonumber
    \end{align}
\end{thm}
\begin{proof}
    We apply Thm.~\ref{thm:CLT} to the following setting. The base category $\Chart$ has finite products by Prop.~\ref{prop:Chartproducts}. The time domains are the hybrid time domain objects $\binom{\Lambda}{\coprod_\Lambda I_j}$ with unit clock coalgebras (Def.~\ref{def:clockcoalgebra}), the plant is the $\H$-coalgebra $\binom{f_d}{f_c}$, and the measurement object is $\binom{\sigma_d}{\sigma_c}$, which we assume satisfies the comparison property (Def.~\ref{def:measurementobject}). 
    The posetal structures on $\T R_c$ (fiberwise) and $\P(R_d \times \underline{R_c})$ (Hoare), as described in Ex.~\ref{ex:posetal2}, equip $\H\binom{R_d}{R_c}$ with the required posetal structure. The positive definiteness conditions~\eqref{eq:pd_c} are directly derived from Def.~\ref{def:positivedefinite} in $\Chart$, using the composition rule given in Def.~\ref{def:chart}.
    Diagram~\eqref{eqn:lyapunov} instantiated in $\Chart$ gives two diagrams. The first says precisely \eqref{eq:flow}. The second reduces to the following:
    \[\begin{tikzcd}
    	{S \times \underline M} & {S \times \underline M \times \underline M} & {R_d \times R_c} \\
    	{\P(S \times \underline M)} & {\P(S \times \underline M \times \underline M)} & {\P(R_d \times R_c)}
    	\arrow["{\id \times \Delta}", from=1-1, to=1-2]
    	\arrow["{f_d}"', from=1-1, to=2-1]
    	\arrow["{V_d\times V_c}", from=1-2, to=1-3]
    	\arrow[Rightarrow, from=1-3, to=2-1]
    	\arrow["{\sigma_d}", from=1-3, to=2-3]
    	\arrow["{\P(\id \times \Delta)}"', from=2-1, to=2-2]
    	\arrow["{\P(V_d \times V_c)}"', from=2-2, to=2-3]
    \end{tikzcd}\]
    For $(s,x) \in S \times \underline M$ this says exactly \eqref{eq:jump}.
    
    Thm.~\ref{thm:CLT} now gives stability: there exists a class $\K$ morphism $\alpha \maps \binom{R_d}{R_c} \Rightarrow \binom{R_d}{R_c}$ such that for any solution, $\|(\gamma_d, \gamma_c)\|_{z^*} \leq \alpha(\|(\gamma_d(0), \gamma_c(0))\|_{z^*})$. Since the semi-norm on $\binom{S}{M}$ decomposes as a product of the continuous semi-norm $\|\cdot\|^c_{z^*}$ on $M$ and the discrete semi-norm $\|\cdot\|^d_{z^*}$ on $S \times \underline{M}$, and $\alpha = \binom{\alpha_d}{\alpha_c}$ acts componentwise, the stability bound separates into the two inequalities in~\eqref{eq:stab_c}.
\end{proof}

For the remainder, we elide writing out positive definiteness conditions and checking them except in the few cases when it is enlightening to do so.

\subsect{Hybrid Periodic Orbits}
Before turning to Zeno stability, we demonstrate Thm.~\ref{thm:CLT_hybrid} on hybrid periodic orbits. The orbit itself is the generalized element in question, and different degrees of stability correspond to different choices of measurement object. In all of them, we use the same posetal object $\binom 1{\Rplus}$ as described in Ex.~\ref{ex:posetal2}. 
The discrete part of $V$ is necessarily trivial. The discrete part of $\sigma$ has the form $\sigma_d \maps \Rplus \to \P(\Rplus)$. The jump condition \eqref{eq:jump} reduces to $V_c(x^+) \leq \sigma_d(V_c(x))$.

\begin{cor}[Stability of Hybrid Periodic Orbits]
\label{cor:orbitstability}
    The following stability results hold for the hybrid periodic orbit $\bar\psi$, each obtained by applying Thm.~\ref{thm:CLT_hybrid} with the indicated measurement object and a function $V$ (or $V_\varepsilon$) positive definite with respect to $\bar\psi$.
    \begin{enumerate}
        \item \textbf{Stability.} Take $\sigma_c = 0$ and $\sigma_d(a) = \{a\}$. If $V$ satisfies:
        \begin{align}
            \frac{dV}{dt}(x) \cdot f_v(x) \leq 0, \quad x \in D_v, \label{eq:stab_flow} \\
            V(R_e(x)) \leq V(x), \quad x \in G_e, \label{eq:stab_jump}
        \end{align}
        then $\bar\psi$ is stable.
        \item \textbf{Asymptotic Stability.} Take $\sigma_c(a) = -\alpha_3(a)$ for some $\alpha_3 \in \mathcal{K}$ and $\sigma_d(a) = \{a-\beta(a)\}$ for some continuous $\beta$ with $0 \leq \beta(r) \leq r$. If $V$ satisfies:
        \begin{align}
            \frac{dV}{dt}(x) \cdot f_v(x) \leq -\alpha_3(V(x)), \quad x \in D_v, \label{eq:asym_flow} \\
            V(R_e(x)) \leq V(x) - \beta(V(x)), \quad x \in G_e, \label{eq:asym_jump}
        \end{align}
        then $\bar\psi$ is asymptotically stable.
        \item \textbf{Rapid Exponential Stability (RES).} Take $\sigma_c(a) = -\frac{c}{\varepsilon}a$ and $\sigma_d(a) = \{\kappa_e a\}$. Suppose the inter-jump durations satisfy $0 < \Delta_{\min} \leq \Delta t_k \leq \Delta_{\max}$ and the period is $T \in [T_{\min}, T_{\max}]$. If $V_\varepsilon$ is positive definite with respect to $\bar\psi$ uniformly in $\varepsilon > 0$ and satisfies:
        \begin{align}
            \frac{dV_\varepsilon}{dt}(x) \cdot f_v(x) \leq -\frac{c}{\varepsilon} V_\varepsilon(x), \quad x \in D_v, \label{eq:RES_flow} \\
            V_\varepsilon(R_e(x)) \leq \kappa_e V_\varepsilon(x), \quad x \in G_e, \label{eq:RES_jump}
        \end{align}
        then $\bar\psi$ is RES if $\Pi_\kappa \leq 1$, or if $\varepsilon < \varepsilon^* := cT_{\min}/\log\Pi_\kappa$ if $\Pi_\kappa > 1$, where $\Pi_\kappa := \prod_{i=0}^{K-1}\kappa_{e_i}$. Moreover,
        \begin{equation*}
            V_\varepsilon(t) \leq Ce^{-\lambda_\varepsilon t} V_\varepsilon(0), \ \lambda_\varepsilon \geq \frac{1}{T_{\max}}\left(\frac{cT_{\min}}{\varepsilon} - \log\Pi_\kappa\right).
        \end{equation*}
    \end{enumerate}
\end{cor}

\begin{figure}
    \centering
    \includegraphics[width=\linewidth]{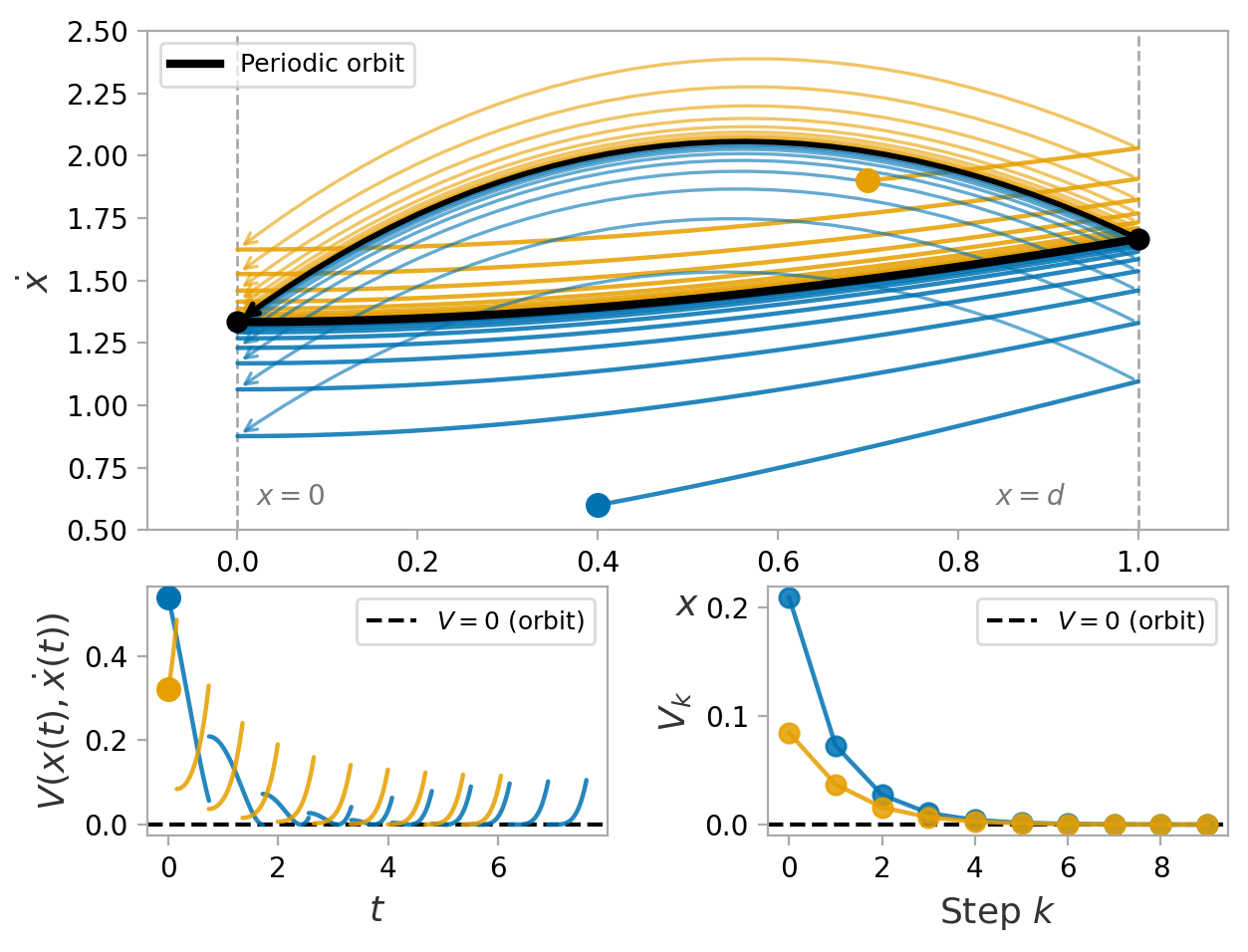}
    \vspace{-6mm}
    \caption{Convergence to a hybrid periodic orbit in a reduced bipedal walking model. Left: phase portrait with periodic orbit (black) and two converging trajectories. Middle: Lyapunov function $V(t) = (\dot x - \dot x^*)^2$ in continuous time. Right: $V_k$ at each step, converging at rate $\lambda^2 = 0.64$.}
    \label{fig:walking}
    \vskip -4mm
\end{figure}

\begin{proof}
    All three parts apply Thm.~\ref{thm:CLT_hybrid} with posetal object $\binom{1}{\R_{\geq 0}}$. Since $R_d = 1$, the Lyapunov morphism is determined by its continuous component, which we denote $V$ (or $V_\varepsilon$ in part~3). Since $f_d(v,x)$ is a singleton on the guard, the Hoare order reduces  condition~\eqref{eq:jump} to $V(R_e(x)) \leq \max \sigma_d(V(x))$.


    \textit{Part~1.} Substituting $\sigma_c = 0$ and $\sigma_d(a) = \{a\}$ into~\eqref{eq:flow}--\eqref{eq:jump} gives~\eqref{eq:stab_flow}--\eqref{eq:stab_jump}. The measurement system has constant solutions, so the comparison property holds trivially.

    \textit{Part~2.} Substituting $\sigma_c(a) = -\alpha_3(a)$ and $\sigma_d(a) = \{a - \beta(a)\}$ gives~\eqref{eq:asym_flow}--\eqref{eq:asym_jump}. Exact solutions of $\sigma$ are non-increasing, and the classical comparison lemma gives $\phi \leq \psi$ for any sub-solution $\phi$, confirming the comparison property. Since $\alpha_3 \in \K$ and $\beta(r) > 0$ for $r > 0$, solutions of $\sigma$ converge to zero, giving asymptotic stability.

    \textit{Part~3.} Substituting $\sigma_c(a) = -\frac{c}{\varepsilon}a$ and $\sigma_d(a) = \{\kappa_e a\}$ gives~\eqref{eq:RES_flow}--\eqref{eq:RES_jump}. The comparison property follows as in Part~2. 
    For Part~3, between jumps $\dot{r} = -\frac{c}{\varepsilon}r$, and at jumps $r^+ = \kappa_{e_k} r^-$, so over one period of $K$ jumps with duration $T$:
    \[
        V_\varepsilon(\tau_K) \leq V_\varepsilon(\tau_0) \cdot e^{-(c/\varepsilon)T} \cdot \Pi_\kappa,
    \]
    where $\Pi_\kappa = \prod_{i=0}^{K-1} \kappa_{e_i}$. The per-period factor $\mu := e^{-(c/\varepsilon)T}\Pi_\kappa$ satisfies $\mu < 1$ iff $(c/\varepsilon)T > \log\Pi_\kappa$, which holds for all $\varepsilon > 0$ when $\Pi_\kappa \leq 1$ and for $\varepsilon < \varepsilon^*$ when $\Pi_\kappa > 1$. Accounting for partial periods, $V_\varepsilon(t) \leq C e^{-\lambda_\varepsilon t} V_\varepsilon(0)$ with $C = \max(1, \max_i \kappa_{e_i})$ and
    \[
        \lambda_\varepsilon = \frac{1}{T}\left(\frac{cT}{\varepsilon} - \log\Pi_\kappa\right) \geq \frac{1}{T_{\max}}\left(\frac{cT_{\min}}{\varepsilon} - \log\Pi_\kappa\right),
    \]
    recovering the RES result of~\cite{RESCLF}.
\end{proof}

\section{Zeno Stability}
%

This section leverages CLT to derive new conditions for the stability of Zeno equilibria and the existence of Zeno behavior.
We introduce a Zeno measurement object yielding stability conditions, then discuss how stability can be transferred between hybrid systems via simulation morphisms. 


We begin by specializing the measurement object to obtain a Zeno stability result. The underlying posetal object is $\binom{\R_{\geq 0}}{\R_{\geq 0}}$ with the pointwise order (Ex.~\ref{ex:posetal2}) and base point $0_R = (0,0)$.

\begin{defn}
\label{def:zenosigma}
    The \define{Zeno measurement object} is the $\H$-coalgebra $\sigma \maps \binom{\R_{\geq 0}}{\R_{\geq 0}} \chart \H\binom{\R_{\geq 0}}{\R_{\geq 0}}$ given by
    \begin{equation}
    \label{eq:sigma}
        \sigma_c(r) = -c, \qquad \sigma_d(a, r) = \begin{cases} \{(\lambda a, a)\} & r = 0, \\ \emptyset & r > 0, \end{cases}
    \end{equation}
    for fixed $c > 0$ and $\lambda \in (0,1)$. Note that $r$ is the continuous variable, and $a$ is the discrete variable.
\end{defn}

\begin{prop}
\label{prop:comparisonminusc}
    The Zeno measurement object (Def.~\ref{def:zenosigma}) satisfies the comparison property \eqref{eq:comparisonproperty}.
\end{prop}
\begin{proof}
    Let $\psi$ be a solution to $\sigma$ and $\varphi$ a sub-solution with $\varphi(0) \leq \psi(0)$. Unpacking the (lax) coalgebra morphism conditions: on each $I_j$, $\dot{\psi}_c = -c$ and $\dot{\varphi}_c \leq -c$; at each jump, $\psi_c = \varphi_c = 0$ (forced since $\sigma_d(a,r) = \emptyset$ for $r > 0$), $\psi_d^+ = \lambda\psi_d$ and $\varphi_d^+ \leq \lambda\varphi_d$, and $\psi_c^+ = \psi_d$ and $\varphi_c^+ \leq \varphi_d$. Between jumps the discrete lax condition is vacuous.

    We show $\varphi \leq \psi$ by induction. On $I_0$, $\dot{\varphi}_c \leq -c = \dot{\psi}_c$ and $\varphi_c(0,\tau_0) \leq \psi_c(0,\tau_0)$, so $\varphi_c \leq \psi_c$ on $I_0$ by the classical comparison lemma. Assuming $\varphi_c \leq \psi_c$ on $I_j$ and $\varphi_d \leq \psi_d$ at time $\tau_j$, the jump gives $\varphi_d^+ \leq \lambda\varphi_d \leq \lambda\psi_d = \psi_d^+$ and $\varphi_c^+ \leq \varphi_d \leq \psi_d = \psi_c^+$, and the comparison lemma extends $\varphi_c \leq \psi_c$ to $I_{j+1}$.
\end{proof}

With the Zeno measurement object \eqref{eq:sigma}, the jump condition \eqref{eq:jump} simplifies considerably. Since $\sigma_d(a, r) = \emptyset$ when $r > 0$, condition \eqref{eq:jump} can only be satisfied at points $(s,x)$ where $f_d(s,x)$ is nonempty if $V_c(x) = 0$, in which case $\sigma_d(V_d(s,x), 0) = \{(\lambda V_d(s,x), V_d(s,x))\}$. The jump condition therefore reduces to three conditions: for all $(s,x)$ with $f_d(s,x) \neq \emptyset$ and all $(s',x') \in f_d(s,x)$:
\begin{align}
    V_c(x) &= 0, \label{eq:jump_vc0} \\
    V_d(s',x') &\leq \lambda \cdot V_d(s,x), \label{eq:jump_vd} \\
    V_c(x') &\leq V_d(s,x). \label{eq:jump_vc}
\end{align}

\begin{cor}\label{cor:zenostability}
    Let $H = (\Gamma, D, G, R, F)$ be a hybrid system with a forward-invariant point $z^* \in \bigsqcup_{v \in V} D_v$. Suppose there exist constants $c > 0$ and $\lambda \in (0,1)$, class $\K$ functions $\underline{\alpha}_1, \overline{\alpha}_1, \underline{\alpha}_2, \overline{\alpha}_2$, and positive definite maps $V_c, V_d \maps \bigsqcup_{v \in V} D_v \to \R_{\geq 0}$ with $V_c$ continuously differentiable, satisfying 
    the flow condition, for all $v \in V$ and $x \in D_v$:
    \begin{align}
        \frac{dV_c}{dt}(x) \cdot f_v(x) &\leq -c, \label{eq:zeno_flow2}
    \end{align}
    and the jump conditions, for all $e \in E$ and $x \in G_e$:
    \begin{align}
        V_c(x) &= 0, \label{eq:zeno_guard2} \\
        V_d(R_e(x)) &\leq \lambda  V_d(x), \label{eq:zeno_jump_Vd2} \\
        V_c(R_e(x)) &\leq V_d(x). \label{eq:zeno_jump_Vc2}
    \end{align}
    Then $z^*$ is stable.
\end{cor}
\begin{proof}
    We verify the hypotheses of Thm.~\ref{thm:CLT_hybrid} with the Zeno measurement object $\sigma$ (Def.~\ref{def:zenosigma}). The posetal object is $\binom{\R_{\geq 0}}{\R_{\geq 0}}$ with the pointwise order and base point $(0,0)$. 
    For the flow condition~\eqref{eq:flow}: $\sigma_c(r) = -c$, so~\eqref{eq:zeno_flow2} gives $\frac{dV_c}{dt} \cdot f_v(x) \leq -c = \sigma_c(V_c(x))$. For the jump conditions: as derived in~\eqref{eq:jump_vc0}--\eqref{eq:jump_vc}, $\sigma_d(a,0) = \{(\lambda a, a)\}$ and the Hoare order decompose~\eqref{eq:jump} into~\eqref{eq:zeno_guard2}--\eqref{eq:zeno_jump_Vc2}. The comparison property holds by Prop.~\ref{prop:comparisonminusc}. Thm.~\ref{thm:CLT_hybrid} gives stability of $z^*$.
\end{proof}

\begin{thm}[Zeno Stability]\label{thm:ZenoStability}
    Under the hypotheses of Cor.~\ref{cor:zenostability}, suppose additionally that $V_d$ is continuously differentiable and satisfies
    \begin{align}
        \frac{dV_d}{dt}(x) \cdot f_v(x) &\leq 0, \qquad x \in D_v. \label{eq:zeno_flow_Vd}
    \end{align}
    Then for any execution $\chi = (\Lambda, I, \rho, C)$ with inter-jump durations $\Delta t_k := \tau_{k+1} - \tau_k$:
    \begin{align}
        \sum_{k \in \Lambda} \Delta t_k \leq \frac{V_c(c_0(\tau_0))}{c} + \frac{V_d(c_0(\tau_0))}{c(1 - \lambda)} < \infty. \label{eq:zenobound}
    \end{align}
    In particular, if $\Lambda = \N$, then $\chi$ is a Zeno execution.
\end{thm}
\begin{proof}
    Define $W \maps \bigsqcup_{v \in V} D_v \to \R_{\geq 0}$ by
    \[
        W(x) := V_c(x) + \frac{V_d(x)}{1 - \lambda}.
    \]
    We show that $W$ is non-negative, decreases at rate at least $c$ during flow, and is non-increasing at jumps.

    \textit{Flow.} For $x \in D_v$:
    \[
        \frac{dW}{dt} \cdot f_v(x) = \frac{dV_c}{dt} \cdot f_v(x) + \frac{1}{1-\lambda}\frac{dV_d}{dt} \cdot f_v(x) \leq -c + 0 = -c,
    \]
    using~\eqref{eq:zeno_flow2} and~\eqref{eq:zeno_flow_Vd}.

    \textit{Jumps.} For $e \in E$ and $x \in G_e$, condition~\eqref{eq:zeno_guard2} gives $V_c(x) = 0$, so $W(x) = V_d(x)/(1 - \lambda)$. After reset:
    \begin{align*}
        W(R_e(x)) 
        &= V_c(R_e(x)) + \frac{V_d(R_e(x))}{1-\lambda} 
        \\&\leq V_d(x) + \frac{\lambda V_d(x)}{1-\lambda} = \frac{V_d(x)}{1-\lambda} = W(x),
    \end{align*}
    using~\eqref{eq:zeno_jump_Vd2} and~\eqref{eq:zeno_jump_Vc2}.

    Since $W \geq 0$ and $\frac{dW}{dt} \leq -c$ during flow while $W$ is non-increasing at jumps, the total flow time satisfies
    \[
        \sum_{k \in \Lambda} \Delta t_k 
        \leq 
        \frac{V_c(c_0(\tau_0))}{c} + \frac{V_d(c_0(\tau_0))}{c(1 - \lambda)} < \infty. \qedhere
    \]
\end{proof}


\begin{example}[Bouncing Ball]
\label{ex:bouncingball}
    A ball of unit mass bounces vertically under gravity with coefficient of restitution $\lambda \in (0,1)$. The state $(x_1, x_2) \in \R_{\geq 0} \times \R$ represents height and vertical velocity respectively. The domain, guard, vector field, and reset map are:
    \begin{align*}
        &D = \{(x_1, x_2) \in \R^2 \mid x_1 \geq 0\}, \\
        &G = \{(x_1, x_2) \in \R^2 \mid x_1 = 0,\ x_2 \leq 0\}, \\
        &f_{\mathrm{ball}}(x_1, x_2) = (x_2,\ -g), \\
        &R_{\mathrm{ball}}(x_1, x_2) = (0,\ -\lambda x_2).
    \end{align*}
    Since $\lambda \in (0,1)$, the ball undergoes infinitely many impacts in finite time, accumulating at the origin. The origin is a forward-invariant point (not an equilibrium of the flow). We apply Thm.~\ref{thm:ZenoStability} with $z^* = (0,0)$.

    Define $V_c, V_d \maps D \to \R_{\geq 0}$ by
    \begin{equation*}
        V_c(x_1, x_2) = c \cdot \tau(x_1, x_2), \ 
        V_d(x_1, x_2) = \tfrac{2c}{g} \cdot \upsilon(x_1, x_2),
    \end{equation*}
    \begin{equation*}
        \tau(x_1, x_2) = \frac{x_2 + \sqrt{x_2^2 + 2gx_1}}{g}, \ \upsilon(x_1, x_2) = \sqrt{x_2^2 + 2gx_1}.
    \end{equation*}
    Here $\tau(x_1, x_2)$ is the time remaining until the next impact and $\upsilon(x_1, x_2)$ is the speed at the next impact. 

    \textit{Flow condition.} Along $f_{\mathrm{ball}}$, we have $\frac{dV_c}{dt} \cdot f_{\mathrm{ball}}(x) = -c$, so \eqref{eq:zeno_flow2} holds with equality.
    
    \textit{Jump conditions.} On the guard $x_1 = 0$, $x_2 \leq 0$, we have $\tau(0,x_2) = (x_2 + |x_2|)/g = 0$, so $V_c = 0$, verifying \eqref{eq:zeno_guard2}. After reset to $(0, -\lambda x_2)$ with $-\lambda x_2 \geq 0$:
    \begin{align}
        V_d(R_{\mathrm{ball}}(x)) &= \frac{2c\lambda}{g}|x_2| = \lambda V_d(x) \leq V_d(x), \label{eq:ball_jump1} \\
        V_c(R_{\mathrm{ball}}(x)) &= c\tau(0,-\lambda x_2) = \frac{2c\lambda}{g}|x_2| = \lambda V_d(x), \label{eq:ball_jump2}
    \end{align}
    verifying \eqref{eq:zeno_jump_Vd2} and \eqref{eq:zeno_jump_Vc2} with equality. By Thm.~\ref{thm:ZenoStability}, the origin is Zeno stable and
    \begin{equation}
        \sum_{k \in \N} \Delta t_k \leq \frac{V_c(x_0)}{c} + \frac{V_d(x_0)}{c(1-\lambda)} = \tau_0 + \frac{2\upsilon_0}{g(1-\lambda)},
    \end{equation}
    where $\tau_0 = \tau(x_0)$ is the time to first impact and $\upsilon_0 = \upsilon(x_0)$ is the speed at first impact.
\end{example}

\begin{figure}[t]
    \centering
    \includegraphics[width=1\linewidth]{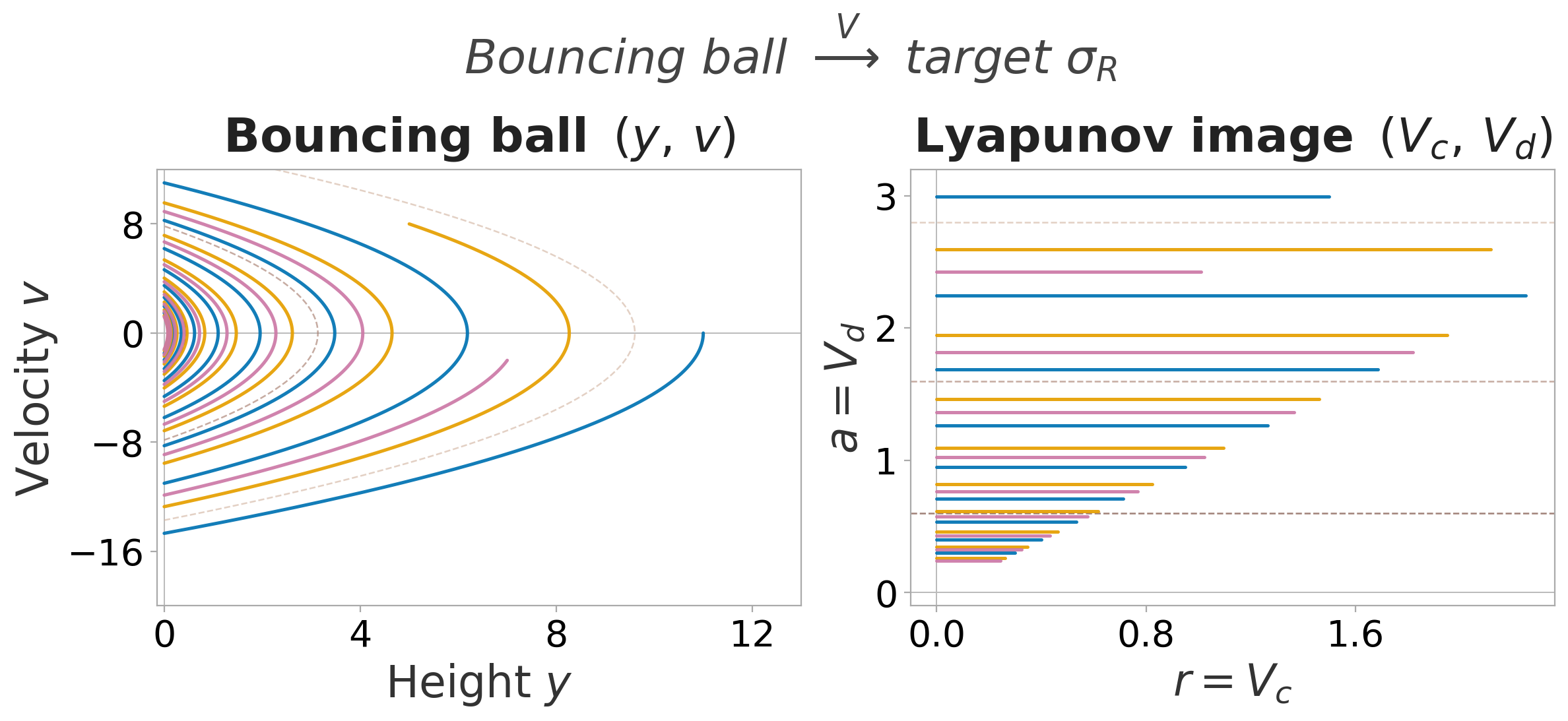}
    \vspace{-6mm}
    \caption{Phase portrait of the bouncing ball (left) and its image under the Lyapunov map $V$ in the target system $\sigma_R$ (right). }
    \label{fig:ball_lyapunov}
    \vskip -4mm
\end{figure}

\subsect{Transferring Stability via Simulation}

\begin{defn}
\label{def:pullbacknorm}
    Let $\Phi = \binom{\Phi_d}{\Phi_c} \maps \binom{S_E}{M_E} \chart \binom{S_Y}{M_Y}$ be a chart, let $d_Y$ be a semi-metric on $\binom{S_Y}{M_Y}$, and let $y^* \maps \binom{Z_d}{Z_c} \chart \binom{S_Y}{M_Y}$ be a generalized element. The \define{pullback semi-norm} along $\Phi$ with respect to $y^*$ is the map $\|\cdot\|_\Phi \maps \binom{S_E}{M_E} \to R$ defined by:
    \[
        \left\|\binom{s}{x}\right\|_\Phi := \left\|\Phi\binom{s}{x}\right\|_{y^*} = \inf_{Z_d \times \underline{Z_c}} d_Y\left(\Phi\binom{s}{x},\ y^*\binom{z_d}{z_c}\right).
    \]
\end{defn}

\begin{rmk}
\label{rmk:pullbackkernel}
    When $\Phi_c$ is injective, $\|\cdot\|_\Phi$ is a valid semi-metric with kernel $\Phi_c^{-1}(y^*(Z_c))$. When $\Phi_c$ is not injective, $\|\cdot\|_\Phi$ is still well-defined but its kernel is larger than the diagonal, corresponding to stability to the set $\Phi_c^{-1}(y^*(Z_c))$. In the Lagrangian hybrid systems application, this set is precisely $Z_h$.
\end{rmk}

\begin{defn}
\label{def:simulationmorphism}
    Let $\binom{f_d}{f_c} \maps \binom{S_E}{M_E} \chart \H\binom{S_E}{M_E}$ and $\binom{\sigma_d}{\sigma_c} \maps \binom{S_Y}{M_Y} \chart \H\binom{S_Y}{M_Y}$ be $\H$-coalgebras. A chart $\Phi = \binom{\Phi_d}{\Phi_c} \maps \binom{S_E}{M_E} \chart \binom{S_Y}{M_Y}$ is a \define{simulation morphism} from $(E, f)$ to $(Y, \sigma)$ if the following two conditions hold. The continuous simulation condition $\frac{d\Phi_c}{dt} \cdot f_c(x) \leq \sigma_{c}(\Phi_c(x))$, and the discrete simulation condition:
    \begin{align*}
        \{(\Phi_d(s', x'), \Phi_c(x')) \mid& (s', x') \in f_d(s, x)\} \nonumber
        \\&\leq \sigma_{d}(\Phi_d(s, x), \Phi_c(x)). \label{eq:simdiscrete}
    \end{align*}
\end{defn}

\begin{thm}[Transference of Stability via Simulation]
\label{thm:transference}
    Let $(E,f)$ be an $\H$-coalgebra, $(Y,\sigma)$ a measurement object with semi-metric $d_Y$ and generalized element $y^*$, $\Phi \maps E \chart Y$ a simulation morphism, and $V \maps Y \chart R$ a Lyapunov morphism for $\sigma$ with respect to $y^*$. Then $W = V \circ \Phi$ satisfies the conditions of Thm.~\ref{thm:CLT_hybrid} for $(E,f)$ and is positive definite with respect to the pullback semi-norm $\|\cdot\|_\Phi$, certifying stability to $\Phi_c^{-1}(y^*(Z_c))$. If $\Phi_c$ is injective, $W$ certifies stability to the single generalized element $z^* = \Phi^{-1}(y^*(Z_c))$.
\end{thm}
\begin{proof}
    By the Chart composition formula, $W_c = V_c \circ \Phi_c$ and $W_d(s,x) = V_d(\Phi_d(s,x), \Phi_c(x))$.

    \textit{Positive definiteness.} Substituting $y = \Phi(s,x)$ into the class $\K$ bounds on $V$ gives $\underline{\alpha}(\|\Phi(s,x)\|_{y^*}) \leq W(s,x) \leq \overline{\alpha}(\|\Phi(s,x)\|_{y^*})$, which is positive definiteness with respect to $\|\cdot\|_\Phi$ (Def.~\ref{def:pullbacknorm}).

    \textit{Lax commutativity.} Diagram~\eqref{eqn:lyapunov} for $W = V \circ \Phi$ factors as:
    \[
        \sigma' \circ V \circ \Phi \geq \H(V) \circ \sigma \circ \Phi \geq \H(V) \circ \H(\Phi) \circ f = \H(V \circ \Phi) \circ f,
    \]
    where the first inequality is the Lyapunov condition on $V$, the second is the simulation condition on $\Phi$, and the equality is functoriality. Thm.~\ref{thm:CLT_hybrid} now applies.
\end{proof}

\begin{figure*}
    \centering
    \includegraphics[width=\linewidth]{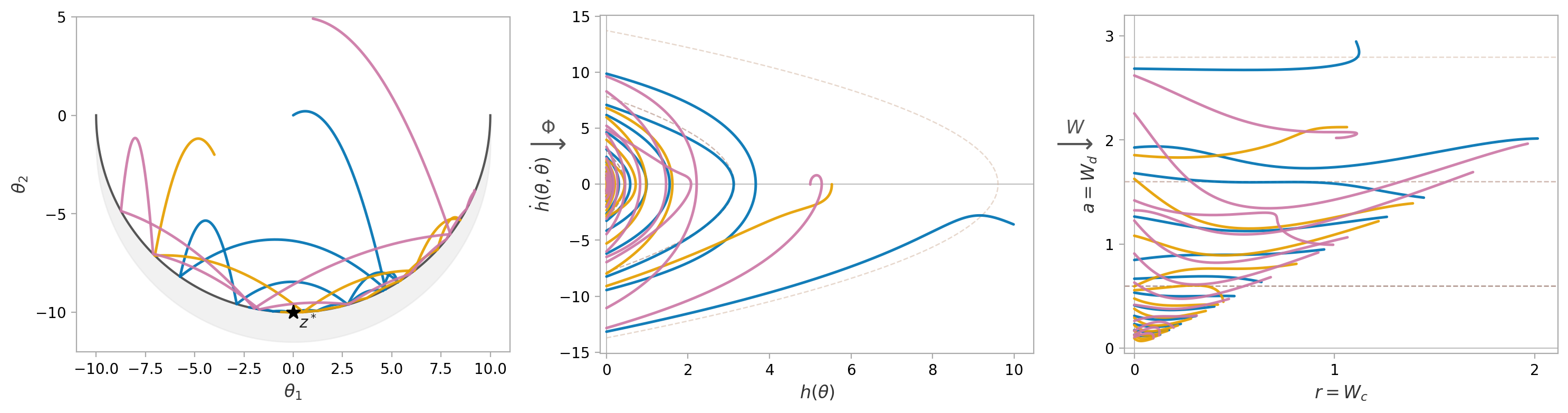}
    \vspace{-8mm}
    \caption{Transference from the bowl system to the bouncing ball via $\Phi$, then to the target $\sigma_R$ via $W$. Left: trajectories in configuration space converging to $z^*$ (star). Middle: the same trajectories in bouncing ball coordinates $(h, \dot h)$. Right: Lyapunov image $(W_c, W_d)$.}
    \label{fig:bowl}
    \vskip -4mm
\end{figure*}

\begin{defn}[Lagrangian Hybrid System {\cite{NonsmoothMechanics}}]
\label{def:LagrangianHybrid}
    A \define{Lagrangian hybrid system} consists of a manifold $\Theta$, a smooth unilateral constraint $h \maps \Theta \to \R$ with $0$ a regular value of $h$, and a Lagrangian $L$ on $\T\Theta$. The state space is $X = \T\Theta$, and the domain, guard, vector field, and reset map are:
    \begin{align*}
        &D_h = \{(\theta, \dot\theta) \in \T\Theta \mid h(\theta) \geq 0\}, \\
        &G_h = \{(\theta, \dot\theta) \in \T\Theta \mid h(\theta) = 0,\ \dot{h}(\theta, \dot\theta) \leq 0\}, \\
        &f_L(\theta, \dot\theta) = (\dot\theta, M(\theta)^{-1} H(\theta, \dot\theta))
        \\&R_h(\theta, \dot\theta) = (\theta,\ \dot\theta^+),
    \end{align*}
    where $f_L$ is the Euler--Lagrange vector field, $M(\theta)$ is the mass matrix, $H(\theta, \dot\theta)$ contains the Coriolis and gravity terms, and the reset is given by the Newtonian impact equation:
    \begin{equation*}
        \dot\theta^+ = \dot\theta - (1+\lambda) \frac{Dh(\theta)\dot\theta}{Dh(\theta)M(\theta)^{-1}Dh(\theta)^\top} M(\theta)^{-1}Dh(\theta)^\top,
    \end{equation*}
    with coefficient of restitution $\lambda \in (0,1)$. This implies $\dot{h}(\theta, \dot\theta^+) = -\lambda\dot{h}(\theta, \dot\theta)$. The set of \define{Zeno equilibria} is:
    \begin{equation*}
        Z_h := \{(\theta, \dot\theta) \in \T\Theta \mid h(\theta) = 0,\ \dot{h}(\theta, \dot\theta) = 0\}.
    \end{equation*}
\end{defn}

Rather than constructing a Lyapunov function for the Lagrangian hybrid system directly, we map to the bouncing ball and transfer stability along this map. Define
\begin{equation*}
    \Phi \maps D_h \to \R_{\geq 0} \times \R, \qquad \Phi(\theta, \dot\theta) = (h(\theta),\ \dot{h}(\theta, \dot\theta)).
\end{equation*}
Since $\ddot{h}(z^*) < 0$ and $\ddot{h}$ is continuous, there exists a neighborhood $U$ of $z^*$ and a constant $\kappa > 0$ such that $\ddot{h} \leq -\kappa$ on $U$. We compose $\Phi$ with the bouncing ball Lyapunov functions, replacing $g$ by $\kappa$, to obtain a Lyapunov morphism for the Lagrangian system directly.

\begin{example}
\label{ex:LagrangianHybrid}
    Define $W_c, W_d \maps D_h \to \R_{\geq 0}$ by
    \begin{equation*}
        W_c(\theta, \dot\theta) = c  \tau_\kappa(\Phi(\theta, \dot\theta)), \ 
        W_d(\theta, \dot\theta) = \frac{2c}{\kappa} \upsilon_\kappa(\Phi(\theta, \dot\theta)),
    \end{equation*}
    \begin{equation*}
        \tau_\kappa(y, v) = \frac{v + \sqrt{v^2 + 2\kappa y}}{\kappa}, \quad \upsilon_\kappa(y, v) = \sqrt{v^2 + 2\kappa y}.
    \end{equation*}

    \textit{Flow condition.} Since $\frac{\partial \tau_\kappa}{\partial v} \geq 0$, the Lie derivative of $W_c$ along $f_L$ is increasing in $\ddot{h}$. The bouncing ball calculation with $\ddot{h} = -\kappa$ gives $\frac{d}{dt}W_c = -c$, so $\ddot{h} \leq -\kappa$ on $U$ gives $\frac{d}{dt}W_c \cdot f_L(\theta, \dot\theta) \leq -c$, verifying \eqref{eq:zeno_flow2}.

    \textit{Jump conditions.} On $G_h$ where $h(\theta) = 0$ and $\dot{h} \leq 0$, we have $\tau_\kappa(0, \dot{h}) = (\dot{h} + |\dot{h}|)/\kappa = 0$, so $W_c = 0$, verifying \eqref{eq:zeno_guard2}. After reset, $\dot{h}^+ = -\lambda\dot{h}^-$, so $\upsilon_\kappa^+ = \lambda\upsilon_\kappa^-$, giving
    \begin{align*}
        W_d(R_h(\theta, \dot\theta)) &= \frac{2c\lambda}{\kappa}\upsilon_\kappa^- = \lambda W_d \leq W_d,  \\
        W_c(R_h(\theta, \dot\theta)) &= c\tau_\kappa(0, -\lambda\dot{h}^-) = \frac{2c\lambda}{\kappa}|\dot{h}^-| = \lambda W_d, 
    \end{align*}
    verifying \eqref{eq:zeno_jump_Vd2} and \eqref{eq:zeno_jump_Vc2} with equality. By Thm.~\ref{thm:ZenoStability}, every Zeno equilibrium $z^* \in Z_h$ satisfying $\ddot{h}(z^*) < 0$ is stable, and for any execution with initial condition $(\theta_0, \dot\theta_0) \in U$:
    \begin{equation*}
        \sum_{k \in \N} \Delta t_k \leq \frac{W_c(\theta_0, \dot\theta_0)}{c} + \frac{W_d(\theta_0, \dot\theta_0)}{c(1-\lambda)} = \tau_{\kappa,0} + \frac{2\upsilon_{\kappa,0}}{\kappa(1-\lambda)},
    \end{equation*}
    where ${\tau_{\kappa,0} = \tau_\kappa(\Phi(\theta_0,\dot\theta_0))}$ and ${\upsilon_{\kappa,0} = \upsilon_\kappa(\Phi(\theta_0,\dot\theta_0))}$. When ${\dim(\Theta) = 1}$, $Z_h$ reduces to a single point and this recovers pointwise Zeno stability. When $\dim(\Theta) > 1$, $Z_h$ is a manifold and we obtain stability to the set of Zeno equilibria.
\end{example}

\bibliographystyle{plain}
\bibliography{references}

\end{document}